\documentclass [12pt] {article}
\author{Matthieu Arfeux, Guizhen Cui}

\usepackage[latin1]{inputenc}%
\usepackage{amssymb,amsfonts,epsfig,amsmath,graphics,graphicx}%
\usepackage{amsthm}%
\usepackage[all]{xy}%
\usepackage{verbatim}%
\usepackage{enumitem}
\usepackage{bm}

\title {Arbitrary large number of non trivial rescaling limits}
\date{}

\newtheorem{theorem} {Theorem}[section]
\newtheorem{proposition}[theorem]{Proposition}
\newtheorem{lemma}[theorem]{Lemma}

\newtheorem{definition}[theorem]{Definition}
\newtheorem{question}[theorem]{Question}

\theoremstyle{plain}
\newtheorem{theoremint} {Theorem}
\newtheorem{corollaryint} {Corollary}

\newtheorem*{definitionint}{Definition}

\newtheorem{Thmbis} {Theorem}

\newtheorem{remark}[theorem]{Remark}

\renewenvironment{proof}{\noindent{\bf Proof. }}{\hfill{$\square$} \vskip.3cm}

\def\cal{\mathcal}

\def\C{{\mathbb C}}
\def\D{{\mathbb D}}
\def\F{{\mathcal F}}

\def\K{{\mathbb L}}
\def\N{{\mathbb N}}
\def\P{{\mathbb P}}

\def\S{{\mathbb S}}

\def\T{{\mathcal T}}
 
\def\V{{\mathcal V}}

\def\Rat{{\rm Rat}}

\def\revTC{{\overline{\bf rev}}}

\def\Oub{{\Pi}}

 \def\epsilon{{\varepsilon}}

\def\card{{\rm card\,}}
\def\deg{{\rm deg}}

\begin{document}

\maketitle

\begin{abstract}
We construct a family of rational map sequences providing an arbitrary large number of independent rescaling limits of non monomial type. From this, we deduce the existence of a family of rational maps providing a non trivial dynamics on the Berkovich projective line over the field of formal Puiseux series.
\end{abstract}



\section{Introduction.}

 Let us denote by ${\S:=P^1(\C)}$ the Riemann sphere. In this paper we are interested in the behavior of the elements of  $\Rat_d$ the set of rational maps $f:{\mathbb S}\to {\mathbb S}$ of exact degree $d$ under the iteration by composition. More precisely, we are interested in the phenomena of existence of rescaling limits detected in \cite{Stim} but defined for the first time in \cite{K2} as follows. 
 \begin{definitionint}
For a sequence of rational maps $(f_n)_n$ in $\Rat_d$, a {rescaling} is a sequence of Moebius transformations $(M_n)_n$ such that there exist $k\in\N$ and a rational map $g$ of degree $\geq 2$ such that
\[M_n\circ f_n^k\circ M_n^{-1}\to g\]
uniformly on compact subsets of $\S$ with finitely many points removed.

If this $k$ is minimal then it is called the {\it rescaling period} for $(f_n)_n$ at $(M_n)_n$ and $g$ a {\it rescaling limit} for $(f_n)_n$.
\end{definitionint}
In that paper Jan Kiwi introduced the notions of independence of rescalings and of dynamical dependence of rescalings below (these notions are not used in this paper).

\begin{definitionint}[Independence of rescalings]\label{IR}
We say that two rescalings $(M_n)_n$ and $(N_n)_n$ of a sequence of rational maps $(f_n)_n$ are independent and write $N_n\sim M_n$ if \[{N_n\circ M_n^{-1}\to \infty}\] in $\Rat_1$. That is, for every compact set $K$ in $\Rat_1$, the sequence $N_n\circ M_n^{-1}\notin K$ for $n$ big enough.
The rescalings are said to be equivalent if $N_n\circ M_n^{-1}\to M$ in $\Rat_1$.
\end{definitionint}


\begin{definitionint}[Dynamical dependence of rescalings]\label{DDR}
Given a sequence $(f_n)_n\in\Rat_d$ and given $(M_n)_n$ and $(N_n)_n$ of period dividing $q$. We say that $(M_n)_n$ and $(N_n)_n$ are dynamically dependent if, for some subsequences $(M_{n_k})_{n_k}$ and $(N_{n_k})_{n_k}$, there exist $1\leq m\leq q$, finite subsets $S_1,S_2$ of $\S$ and non constant rational maps $g_1,g_2$ such that
 \[L^{-1}_{n_k}\circ f^m_{n_k}\circ M_{n_k}\to g_1\]
uniformly on compact subsets of $\S\setminus S_1$ and
\[M^{-1}_{n_k}\circ f^{q-m}_{n_k}\circ L_{n_k}\to g_2\]
uniformly on compact subsets of $\S\setminus S_2$.
\end{definitionint}

Jan Kiwi proved the two following results.

\begin{Thmbis}\label{alpha} \cite{K2} For every sequence in $\Rat_d$ for $d\geq 2$, there are at most $2d-2$ classes of dynamically independent rescalings with a non postcritically finite rescaling limit. 
\end{Thmbis}

\begin{Thmbis}\label{beta} \cite{K2} For every sequence in $\Rat_2$, there are at most $2$ classes of dynamically independent rescalings. 
\end{Thmbis}

Following these result, J.Kiwi wrote in \cite{K2} the following natural question:

\begin{question} Is there a bound on the number of classes of dynamically independent and non-monomial rescalings that a sequence in $\Rat_d$ can have that would depend only on the degree $d$?
\end{question}
 
In this paper we prove that the answer is no. More precisely we prove the following theorem:

\begin{theoremint}\label{TOF}
For every $n\in\N^*$ and $d\geq3$ there exists a sequence of rational maps $(f_k)_k$ in $\Rat_d$ with $n$ dynamically independent, post-critically finite and non-monomial rescaling limits.
\end{theoremint}


Theorem \ref{alpha} and Theorem \ref{beta} have been proven using non-archimedean dynamics tools and later re-proven in \cite{A1} and \cite{A2} respectively, using a different approach based on the Deligne-Mumford compactification of the moduli space of stable punctured spheres. This compactification has been restated with the language of trees of spheres in \cite{A3}. We suppose that the reader has a good knowledge of the definition provided in \cite{A1}. With this vocabulary, Theorem \ref{TOF} can be written the following way:

\begin{theoremint}\label{TOF2}
For every $n\in\N^*$ and $d\geq3$ there exists a dynamical system between trees of spheres of degree $d$ limit of dynamically marked rational maps, with $n$ cycles of spheres whose associated cover is post-critically finite and non-monomial.
\end{theoremint}

In fact these two approaches are related and their relation has been sketched in \cite{A5}. The results of \cite{A1} and \cite{A2} are a translation of results in \cite{K2}. The authors use dynamical systems between trees of spheres in one case or dynamics on Berkovich spaces in the other, in order to deduce results in holomorphic dynamics.
The approach here is the exact reverse. 

We are inspired by the holomorphic dynamical notion of ``Shishikura trees".
 Those have been introduced by M. Shishikura in \cite{S1} and \cite{S2} for a special case and developed as a general idea during his talks. Following these ideas, \cite{CP} define another notion of Shishikura trees that we will use for our purpose. Our result will be based on the constructions of self-graphting given in \cite{CP}.
 
Denote by $\D^\star$ the punctured unit disc of $\C$. In the analytic context,  rescalings, rescaling limits, independence of rescalings, dynamical dependance of rescaling limits  are also defined by replacing $n\in \N$ tending to infinity by $t\in\D^\star$, $n\to\infty$ by $t\to0$ in the previous definitions, and requiring holomorphic dependence on $t$ (see \cite{A5}).
Jan Kiwi proved the following.

\begin{proposition}[\cite{K2}Proposition 6.1]\label{propKiwi}
Consider a sequence of degree $d$ rational maps ${f_n}$. Let $N \in \N$ and assume that for all $j = 1,...N$ the sequence $(M_{j,n})_n$ is a rescaling of period $q_j$ for $(f_n)_n$ with rescaling limit $g_j$. Then there exists a degree $d$ holomorphic family $({f_t})_t $ and, for each $j = 1,...N$, a holomorphic family of Moebius transformations $(M_{j,t})_t $ such that $(M_{j,t})_t $ is a rescaling for $({f_t})_t $ of period $q_j$ and limit $g_j$.

If $(M_{j,t})_t$ and $(M_{k,t})_t$ are dynamically dependent for $({f_t})_t$, then $(M_{j,n})_n$ and $(M_{k,n})_n$ are dynamically dependent for $({f_n})_n$.
\end{proposition}

Hence Theorem \ref{TOF} has the following consequence.

\begin{corollaryint}\label{resc}
For every $n\in\N^*$ and $d\geq3$ there exists a holomorphic family $f_t\in\Rat_d$ for $t\in \D^\star$ with $n$ dynamically independent, post-critically finite and non-monomial rescaling limits.
\end{corollaryint} 
Using the bridge between non-archimedean dynamics and the dynamics on trees of spheres made explicit in \cite{A5}, we deduce Corollary \ref{Berk} below.

Denote by $\K$ the completion of the field of formal Puiseux series over $\C$ equipped with its usual non-archimedean norm. Denote by $\P^1_{Berk}$ the Berkovich projective line over $\K$ and recall the the type II points of $\P^1_{Berk}$ are the points separating at least three different branches.

\begin{corollaryint}\label{Berk}
For every $n\in\N^*$ and $d\geq3$ there exists $f\in\K(Z)$ of degree $d$ with $n$ periodic type II points in disjoint cycles for the induced dynamics in $\P^1_{Berk}$ whose reduction is post-critically finite and non-monomial.
\end{corollaryint}
 
 We will not recall more details about the non-archimedean dynamics that the interested reader can find in \cite{K2} (or \cite{A5} for a first reading).
 

\bigskip
\noindent{\textbf{Outline.}}

In Section \ref{chap2}, we recall the notion of Shishikura trees from \cite{CP} with some small adaptations. Section \ref{selfgraftingue} deals with the self-grafting construction. In Section \ref{chap3}, we explicit the relation between Shishikura trees and dynamical systems between trees of spheres. In Section \ref{sectionend}, we prove Theorem \ref{TOF} and Theorem \ref{resc}. 
The article ends on a little appendix with some technical results on stable trees.


\bigskip
\noindent{\textbf{Acknowledgments.}} The authors want to thank Tan Lei for presenting the one to the other and for inviting us to discuss in Angers in LAREMA. 


\section{Construction of a Shishikura tree}\label{chap2}

In this section we recall the constructions and results from \cite{CP} with a few changes on some definitions. We point out some properties needed in the next sections. 

\subsection{Canonical multicurve}

Here $f$ denote a hyperbolic rational map.
We denote by ${\cal P}_f$ its post-critical set and ${\cal P}'_f$ the set of accumulation points of ${\cal P}_f$.

Suppose that $E\subset\S$ is a connected set which is neither open or closed. We say that $E$ is {\bf disc-type} if $E$ is contained in a disk $D$ with $\card D\cap {\cal P}_f\leq1$. We say that $E$ {\bf annular-type} if $E$ is not disk-type and is contained in an annulus $A$ with $\card A\cap {\cal P}_f=0$. If $E$ is neither disk-type nor annular-type, then we say that $E$ is {\bf complex-type}.

Recall that given a set $X\subset\S$, a {\bf multicurve} on $\S\setminus X$ is a collection of Jordan curves $\gamma$ in $\S\setminus X$ pairwise disjoint and non isotopic and non-peripheral (i.e. every connected component of $\S\setminus\gamma$ contains at least two points of ${\cal P}_f$). A multicurve $\Gamma_f$ is {\bf totally stable}  (by $f$) if each non-peripheral curve of $f^{-1}(\gamma)$ for $\gamma\in\Gamma_f$ is isotopic rel ${\cal P}_f$ to a curve in $\Gamma_f$ and  if each curve $\gamma\in\Gamma_f$ is isotopic rel ${\cal P}_f$ to a curve in $f^{-1}(\gamma')$ for some curve $\gamma'\in \Gamma_f$.

 \begin{theorem}
 There exist a totally stable multicurve $\Gamma_f$ in $\S\setminus {\cal P_f}$ such that $\forall D\in\S\setminus \Gamma_f$, D contains a unique complex type Julia or is contained in a complex-type Fatou domain, and conversely.
\end{theorem}

 These properties are stable under isotopy rel ${\cal P}_f$ and the isotopy class of such a curve is called the {\bf canonical multicurve} of $f$. As a direct consequence of these definitions we have the following property:

\begin{lemma}\label{stability}
Let $\Gamma_f$ be a representative of the canonical multicurve and  $U$ be a connected component of $\S\setminus \Gamma_f$. Denote by $cc\partial U$ the collection of connected components of $\S\setminus\partial U$. Then we have:
\begin{itemize}
\item $\forall V\in cc\partial U, \card V\cap {\cal P}_f\geq 1$ and
\item $\card (cc\partial U)+\card (U\cap {\cal P}_f) \geq 3.$
\end{itemize}
\end{lemma}

Given a rational map $f$, Jordan curve is called a KB curve if it is a connected component of some Koenig or Böttcher coordinates level curve (cf \cite{DynInOne} for the definitions of the Koenigs and Böttcher coordinates). Similarly, a multicurve is KB if it consists of KB curves. By construction in \cite{CP}, we have:

\begin{proposition}\label{canncurve}
Every canonical multicurve has a representative consisting of a KB multicurve.
\end{proposition}

Such a representative is called a {\bf KB canonical multicurve} of $f$.


\subsection{Shishikura trees}

In this section, $\Gamma_f$ denotes a KB canonical multicurve of $f$ (cf Proposition \ref{canncurve}). 

Given a collection $\Gamma$ of disjoint Jordan curves on $\S$, we define the dual tree $\T_\Gamma$ of $\Gamma$ to be the tree whose vertices consist of the connected components of $\S\setminus \Gamma$ and whose edges are the elements of $\Gamma$ and join two vertices if and only if it lies in both of their boundary.

In \cite{CP}, the {\bf Shishikura tree} $\T_f$ is defined to be the dual tree of $\Gamma_f$. 
Let $\Gamma^{-1}_f:=f^{-1}(\Gamma_f)$ and $\T^{-1}_f$ be its associated dual tree. Define ${\cal P}^{-1}_f:=f^{-1}({\cal P}_f)$. The map $f$ induces a tree map $\tau_f:\T^{-1}_f\to\T_f$ (indeed, it maps adjacent vertices to adjacent vertices and edges to edges, cf Figure \ref{Fig0}).

\begin{remark} This map is slightly different than the one in \cite{CP} where the authors consider $\Gamma_1$ the subset of $\Gamma^{-1}_f$ consisting of its non peripheral elements in $\S\setminus{\cal P}_f$ to define a tree $\T_1$ and then a map $\tau_1:\T_1\to\T_f$. There is a natural map $\pi_1:\T^{-1}_f\to\T_1$ consisting in the inclusion. The relation between $\tau_f$ and $\tau_1$ is provided by the following commutative diagram:
$$\xymatrix{
  \T^{-1}_f \ar[rr]^{\pi_1} \ar[rd]_{\tau_f}  && \T_1 \ar[ld]^{\tau_1} \\
   &\T_f& 
  }$$
  and this change of definition, necessary here, does not affect the results cited from \cite{CP}.
\end{remark}

With our definition, the map $f$ induces also a degree function on the set of edges and vertices that we will denote by $\deg_f$. 
From the total stability of $\Gamma_f$ we deduce a natural identification of $\T_f$ in $\T_f^{-1}$ and the map $\tau_f$ induces dynamics on $\T_f$.



\section{Self-grafting}\label{selfgraftingue}

In \cite{CP} the authors constructed a sequence of rational maps $(f_n)_n$ using {\bf self-grafting}. We recall here what this procedure allowed to construct.

In  \cite{God}, Sébastien Godillon considered the following family of rational maps
$$f_0(z):=\frac{(1-\lambda)[(1-4\lambda+6\lambda^2-\lambda^3)z-2\lambda^3]} {(z-1)^2[(1-\lambda-\lambda^2)z-2\lambda^2(1-\lambda)]}$$
for $\lambda\neq0,1$. We fix some parameter $\lambda_0\neq0,1$ and denote by $f_0$ the corresponding element in this family.
Figure \ref{Fig0} shows the Shishikura trees $\T_{f_0}^{-1}$ and $\T_{f_0}$. 
In order to simplify the notations let us respectively denote $\T_{f_0}$ and $\tau_{f_0}$ by $\T_0$ and $\tau_0$.

 \begin{figure}
  \centerline{\includegraphics[width=11cm]{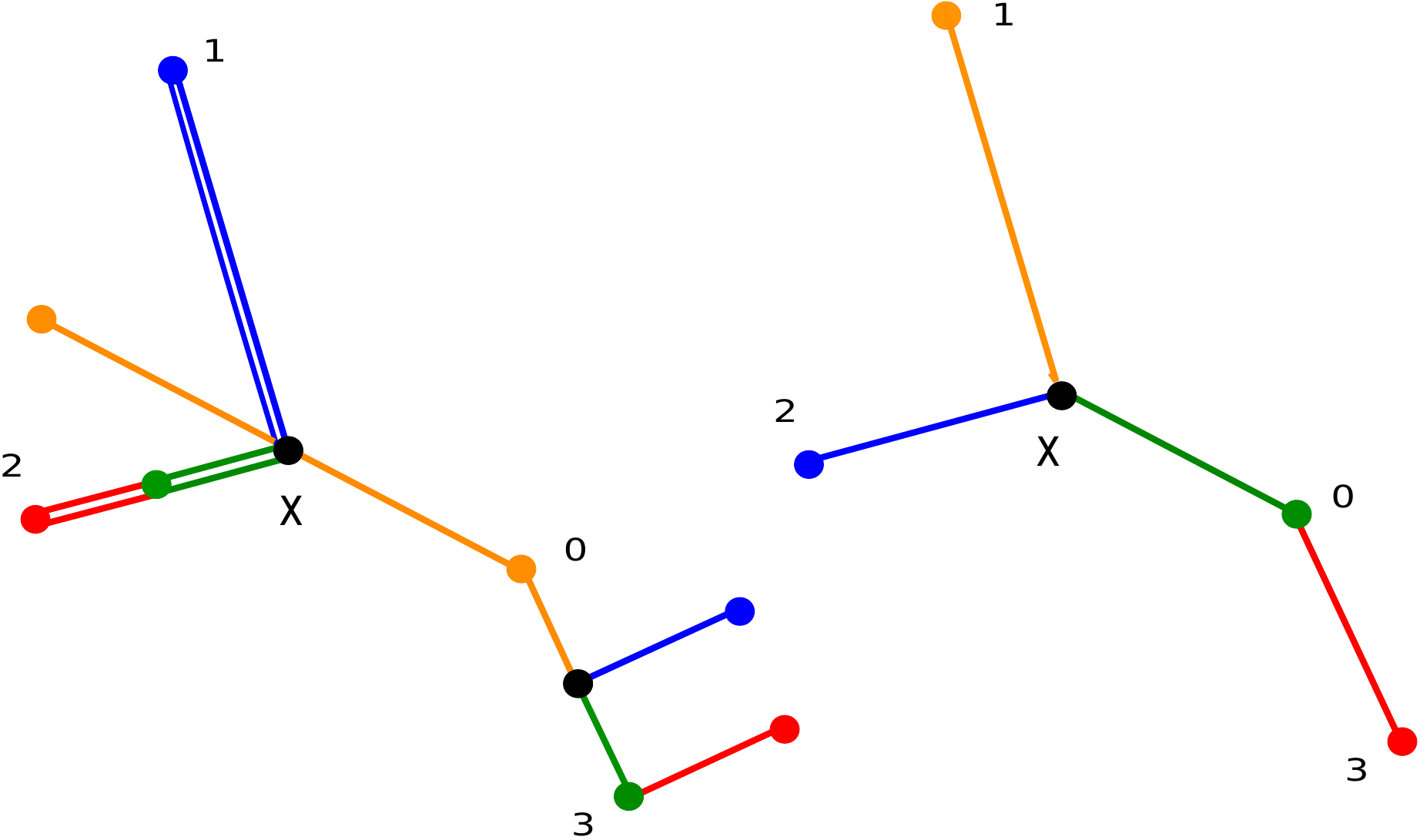}} 
   \caption{The Shishikura trees $\T_{f_0}^{-1}$ (   left) and $\T_{f_0}$ (right) for the function $f_0$ introduced in Section \ref{selfgraftingue}. The map $\tau_f$ maps the vertices and edges to the ones of identical colors. The vertex $x$ is fixed and the numbered vertices form a cycle of period $4$.}
\label{Fig0} \end{figure}

The self-grafting is an induction procedure that from some Shishikura tree $\T_{n}$ uses Thurston's realization theorem to construct 
\begin{itemize}
\item a new tree $\T_{{n+1}}$,
\item a map $\tau_{n+1}$ acting on it, and
\item a hyperbolic rational map ${f_{n+1}}$ such that $\T_{{n+1}}=\T_{f_{n+1}}$ and $\tau_{n+1}=\tau_{f_{n+1}}$.  
\end{itemize}

In addition we have $\T_{n}\lhd\T_{{n+1}}$ in the sense of the following definition.
\begin{definition}\label{defcompatib}
We say that a tree $T^X$ is compatible with a tree $T^Y$ and write $T^X\lhd T^Y$ if 
\begin{itemize}
\item  the vertices of $T^X$ are vertices of $T^Y$ and
\item for all vertices $v$, $v_1$, $v_2$ and $v_3$ of $T^X$, the vertex $v$ separates $v_1$, $v_2$ and $v_3$ in $T^X$ if and only if it does the same in $T^Y$. 
\end{itemize}
\end{definition}
(This definition can be seen as the combinatorial version of the notion of ``subtree" for the embedded real trees.)

Hence edges of $\T_{n}$ between two vertices can be identified as the arc of vertices and edges separating the two corresponding vertices in $\T_{{n+1}}$  (cf Figure \ref{Fig1}).
As a consequence $\T_{0}\lhd\T_{{n}}$. 

Let us recall here some relations between the tree $\T_{{n+1}}$ and $\T_{{n}}$.
We denote by $x_0$ and $x_1$ the two vertices of the $\T_0$ represented on Figure \ref{Fig1}. As $\T_{0}\lhd\T_{{n}}$, the vertices $x_0$ and $x_1$ are also vertices of $\T_n$.  We will denote by $[x_0,x_1]_n$ the arc in $\T_n$ joining those.
The vertices of $\T_{{n+1}}$ lying in the edges of $\T_{{n}}$ form a cycle of period $k>2$. 
Let $v_0\in\T_{{n+1}}$ be the vertex on $[x_0,x_1]_{n+1}\setminus\{x_0\}$ which is the closest to $x_0$.
Denote by $B_0$ the branch of $\T_{{n+1}}$ at $v_0$ that contains the vertex $x_0$ (by the choice of $v_0$, this branch contains only vertices of $\T_{{n}}$). The tree $\T_{n+1}$ is obtained from its elements identified with $\T_n$ by attaching at each point $v_{i+1}:=\tau_{{n+1}}^i(v_i)$ for $i=0..k-1$ a copy of $B_0$ that we denote by $B_{i+1}$. 

Denote by $\iota_{n+1}:\T_{n+1}\to \T_{n+1}$ the involution that exchanges the two branches $B_0$ and $B_k$ via their natural identification. We extend the map $\tau_n$ on $\T_{n+1}$ into a map $f_{n+1}$ that maps $B_k$ to $B_{1}$ and $B_i$ to $B_{i+1}$ for $i=1..k-1$ via the natural identification. The map $\tau_{{n+1}}$ is defined to be $\iota_{n+1}\circ f_{n+1}$.
The degree of $\tau_{{n+1}}$ is the one of $\tau_{0}$ on its elements identified with $\T_{0}$ and it is $1$ elsewhere.

\begin{remark}\label{fundrem}
 With these properties we can point out in particular that 
\begin{itemize}
\item $v_0\in[x_0,x_1]_{n+1}$ is $k$ periodic for the map $\tau_{n+1}$ with some $k>1$,
\item $\tau^k_{{n+1}}$ maps $B_k$ with degree $1$ to $B_0$, and
\item $\tau_{{n+1}}$ maps $B_0$ with degree $2$ to its image.
\end{itemize}
\end{remark}

 \begin{figure}
  \centerline{\includegraphics[width=13cm]{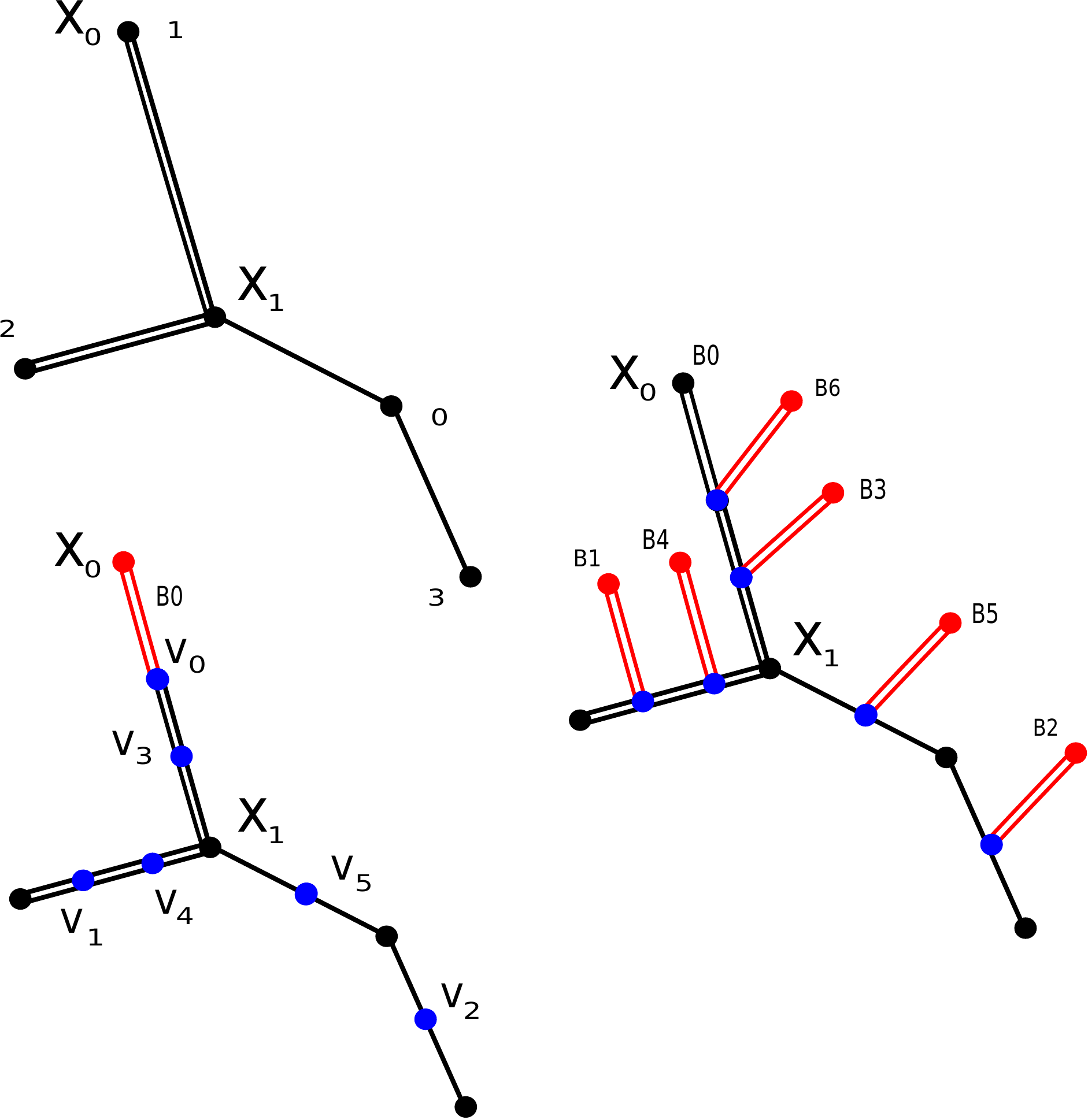}} 
   \caption{On the top left is the tree $\T_{f_0}$ with its cycle as on Figure \ref{Fig0}. The tree lower left is the tree $\T_{f_1}$ restricted to its elements identified with $\T_{f_0}$ and in red is represented the branch $B_0$ as labeled in Section \ref{selfgraftingue}. The $v_i$ form a cycle. On the right is the tree $\T_{f_1}$ and in red are the copies of the branch $B_0$ from the previous tree.}
\label{Fig1} \end{figure}


\section{Jumping between trees}\label{chap3}

In this section $f$ denotes a hyperbolic rational map. We consider a Shishikura tree $\T_f$  and denote by $\tau_f:\T_f^{-1}\to\T_f$ the associated Shishikura tree map. We explain how we associate to $\tau_f$ a dynamical system between trees of spheres which is a limit of dynamically marked rational maps in the sense of \cite{A1}.

\subsection{Notation for trees of spheres}\label{notationTOF}

A tree of spheres marked by a finite set $X$ is usually denoted by $\T^X$. It is a combinatorial tree that we denote by $T^X$ whose set of leaves is $X$ together with the data for each internal vertex $v\in T^X$ of a conformal sphere denoted by $\S_v$ and a different attaching point $i_v(e)\in\S_v$ for every edge $e$ adjacent to $v$. We denote by $X_v$ the set of attaching points of edges on $\S_v$.

If $\T^Y$ is another tree of spheres, we write $\T^X\lhd\T^Y$ if $T^X\lhd T^Y$ and if the identified edges are attached at the same place.

A cover ${\cal F}:{\cal T}^Y\to {\cal T}^Z$ between two trees of spheres marked by finite sets $Y$ and $Z$ is the following data
\begin{itemize}
\item a map $F:T^Y\to T^Z$ mapping leaves to leaves, internal vertices to internal vertices, and edges to edges, 
\item for each internal vertex $v$ of $T^Y$, an holomorphic ramified cover ${f_v:{\S}_v\to {\S}_{F(v)}}$ that satisfies the following properties: 
\begin{itemize}
\item the restriction $f_v : {\S}_v-Y_v\to {\S}_{F(v)}-Z_{F(v)}$ is a cover, 
\item $f_v\circ i_v = i_{F(v)}\circ F$,
\item if $e$ is an edge between $v$ and $v'$, then the local degree of $f_v$ at $i_v(e)$ is the same as the local degree of $f_{v'}$ at $i_{v'}(e)$. 
\end{itemize}
\end{itemize}

For such a $\F$, we say that $(\F,\T^X)$ is dynamical system between trees of spheres if $\T^X\lhd\T^Y$ and $\T^X\lhd\T^Z$.

\subsection{Construction of a tree of spheres}

\subsubsection{Marking sets}

It follows from Lemma \ref{stability} that one can associate to every vertex $v$ of $\T_f$ a set $X_v\subset {\cal P }_f$ of cardinal $3$ such that two points of $X_v$ lie in the same connected component $U$ of $\S\setminus \Gamma_f$ if and only if $U=v$.
Define $X$ to be the union (maybe not disjoint) of the $X_v$ and the set of critical values of $f$. Let $Z:=X\cup f(X)$ and let $Y:=f^{-1}(Z)$. Note that by choice we have $X\subset Y\cap Z$.

We define ${\bf F}$ to be the pair consisting of $F:=(f|_Y:Y\to Z)$ and the degree function $\deg_F:Y\to\N$ corresponding to the local degree of ${f|_Y}$. This ${\cal F}$ is a portrait in the sense of \cite{A1}.

Denote by $i$ and $j$ the respective inclusions of $Y$ and $Z$ into $\S$.
The map $(f,i,j)$ is a rational map dynamically marked by $({\bf F}, X)$ in the sense of \cite{A1}, i.e. $y$ and $z$ are injective and we have the following commutative diagram : 
$$\xymatrix{
    X\ar[r]\ar[rd]&Y \ar[r]^i \ar[d]_{ F}  & \S \ar[d]^{f} \\
    &Z \ar[r]_j & \S
  }$$
 with $\deg_f({i(y)}) = \deg_F(y)$ for all $y\in Y$ and $i|_X=j|_X$.


\subsubsection{Stretching}

Recall that given a set $X'\subset\S$, any disjoint union of curves ${\Gamma'\subset \S\setminus X'}$ gives a natural partition of $X'$ that corresponds to the set of non empty intersections of $X'$ with the different connected components of $\S\setminus \Gamma'$, we denote it by ${\cal P}_{\Gamma',X'}$.

For every $\gamma\in\Gamma$, let us fix an open annulus $A^{\gamma}$ in $\S\setminus X$ that retracts to $\gamma$ and made of KB level curves. 
Let $$\C_r:=\{c\in\C\;|\;\Re(c)>0\}.$$

\begin{proposition} \label{quasidef}
There exists a holomorphic family of rational maps $f_t$ for $t\in\C_r$ and a holomorphic motion 
$$\Phi = \left|
    \begin{array}{ll}
        \C_r\times\S\to&\S \\
        (t,z)\longmapsto &\Phi_t(z)
    \end{array}
\right.
\quad
\text{such that}\quad 
\xymatrix{
    \S\ar[r]^f\ar[d]_{\Phi_t}&\S\ar[d]^{\Phi_t}  \\
   \S \ar[r]_{f_t} & \S,
  }$$
\begin{itemize}
\item ${\cal P}_{\Phi_t(\Gamma),\Phi_t(X)}={\cal P}_{\Gamma,X}$, and
\item for every $\gamma\in\Gamma_f$, $\rm{Modulus} (\Phi_t(A^\gamma)\to\infty$ when $t\to0$.
\end{itemize}
\end{proposition}

\begin{proof}
This construction uses the standard "stretching" deformation described for example in \cite{Nuria}. For our purpose, we choose a fundamental annulus in each of the grand orbit of the Fatou components that contain one of the KB curve. By stretching the complex structure on these annuli by making them going to infinity in modulus when $t$ approaches $0$, we get such a family of rational maps $f_t$ and a holomorphic motion: 
$$\Phi = \left|
    \begin{array}{ll}
       \C_r\times\S\to&\S \\
        (t,z)\longmapsto &\Phi_t(z)
    \end{array}
\right.
$$
such that 
$$\xymatrix{
    \S\ar[r]^f\ar[d]_{\Phi_t}&\S\ar[d]^{\Phi_t}  \\
   \S \ar[r]_{f_t} & \S.
  }$$
 
We choose for every $\gamma\in\Gamma_f$, an open annulus $A^\gamma$ included in the Fatou set of $f$ and that contain $\gamma$. By the stretching deformation gives that $\Phi_t(A^\gamma)$ is still an annulus whose modulus tends to infinity as $t\to0$. Moreover, the $f_t$ are holomorphic so the postcritical set moves holomorphicaly with the parameter $t$ so the sumption about the partitions follows.

\end{proof}

 We define such a $\Phi$ and take a sequence $t_n\in\C_r$ converging to $0$. Then according to \cite{A3}(Theorem 3), after passing to a subsequence, we have the following:
 
 \begin{proposition}\label{compactcov}
 The sequence $(f_{t_n},\Phi_{t_n}\circ i,\Phi_{t_n}\circ j)$ converges dynamically to a dynamical system between trees of spheres $(\F:\T^Y\to\T^Z,\T^X)$ marked by ${\bf F}$.
 \end{proposition}


\subsection{From Shishikura trees to trees of spheres}

Define $T_f$ to be the combinatorial tree of $\T_f$ with the additional vertices $X$ and edges between any element $x\in X$ and the vertex corresponding to the connected component of $\S\setminus \Gamma_f$ containing $x$. Similarly, from $\T_f^{-1}$ we construct $T_f^{-1}$ whose set of leaves is $f^{-1}(X)$. It is clear from the definitions that we have $\T_f\lhd T_f$ and $\T^{-1}_f\lhd T^{-1}_f$.

In this section we prove the following statement:

\begin{theorem}\label{translation}
After changing the labeling of the internal vertices of $T_f$ we have $T_f= T^X$ and ${T_f^{-1}  \lhd T^Y}$. The map $F:T^Y\to T^Z$ restricted to $T_f$ is the map $\tau_f$ and for every attaching point $z$ of an edge $e\in T^X$ on a sphere associated to a vertex $v\in T^X$, the local degree of $ f|_v$ at $z$ is the degree of $\tau_f$ at $e$.
\end{theorem}


\subsubsection{Stable trees and compatibility}
 
 A tree is {\bf stable} if any internal vertex is adjacent to at least three edges.
 
 \begin{remark}\label{stabShishikura}
As a direct consequence of Lemma \ref{stability}, the trees $T_f$ and $T_f^{-1}$ are stable. Also according to  \cite{A3}, the combinatorial trees $T^X,T^Y$ and $T^Z$ associated respectively to $\T^X,\T^Y$ and $\T^Z$ of Proposition \ref{compactcov} are stable.
\end{remark}

In this subsection we prove the following lemma.

\begin{lemma}\label{corostabpart}\label{lemtech}
If $T_1$ and $T_2$ are two stable trees sharing the same set of leaves and such that $T_1\lhd T_2$, then $T_1= T_2$.
\end{lemma}

For any choice of three different leaves in a tree, there exists a unique vertex of this tree separation them.
Given an internal vertex $v$ of a combinatorial tree $T$ and a subset $X_0$ of leaves, there is a natural partition of $X_0$ associated to $v$ consisting of the collections of non empty intersections of $X_0$ with the different connected components of $T\setminus\{v\}$.

\begin{lemma}\label{compartition}
Let $T_1$ and $T_2$ be combinatorial stable trees. Denote by $X_1$ the set of leaves of $T_1$. Suppose that $X_1$ is included in the set of leaves of $T_2$ and that for every $x_1, x_2,x_3\in X_1$ distinct, the partitions of $X_1$ associated to the vertices separating $x_1, x_2,x_3$ are the same.
Then, after relabeling the internal vertices of $T_1$ or $T_2$, we have $T_1\lhd T_2$.
\end{lemma}

\begin{proof}Take an internal vertex $u$ of $T_1$. Then, by stability, this vertex separate three distinct elements $x_1, x_2,x_3$ of $X_1$. Consider the vertex $v$ of $T^2$ separating the same elements. Then, the vertices are considered modulo relabeling, we can suppose that $u=v$. According to these hypothesis, this relabeling can be made consistently on for all the triples of $X_1$, so for all the vertices of $T_1$. It follows that we can consider that the vertices of $T_1$ are vertices of $T_2$.

Consider four vertices $u,u_1,u_2,u_3$ of $T_1$. Suppose that $u$ separates $u_1,u_2$ and $u_3$ in $T_1$. Chose an element $x_1$ (resp. $x_2,x_3$) in the branch on $v$ containing $v_1$ (resp. $v_2,v_3$). Then, in the tree $T_2$, $v$ is on the branch on $v_1$ (resp. $v_2$, then $v_3$) containing $x_2, x_3$ (resp. $v_1,v_3$, then $v_1,v_2$). It follows that $v$ separate $v_1,v_2,v_3$ in $T_2$. The same argument prove the converse property.
\end{proof}

\begin{proof}[Lemma \ref{lemtech}] As $T_1\lhd T_2$, the set of vertices of $T_1$ is included in the one of $T_2$. Any internal vertex $v_2$ of $T_2$ separates three leaves so, because these leaves are also leaves of $T_1$, there is a unique vertex $v_1$ of $T_1$ separating them and it follows that $v_1=v_2$. Hence $T_1$ and $T_2$ have the same set of vertices. From this we deduce that $T_2\lhd T_1$ and it is easy to check that $T_1= T_2$.
\end{proof}


\subsubsection{Proof of $T_f= T^X$ and ${T_f^{-1}  \lhd T^Y}$}

\begin{lemma}\label{equalTXTF}
Up to relabeling the internal vertices of $T^X$, we have ${T_f= T^X.}$ 
\end{lemma}

\begin{proof} Consider a vertex $v\in\T_f$. We deduce from the definition of $X$ that $v\in T^X$ separates three vertices $x_1,x_2$ and $x_3$ which are elements of $X$. Let us relabel the vertex in $T^X$ separating the same three elements in $T^X$ by $v$. Then consider a projective chart $M_{t_n}:\S\to\hat\C$ that maps $\Phi_{t_n}(x_1),\Phi_{t_n}(x_2)$ and $\Phi_{t_n}(x_3)$ to $0,1$ and $\infty$.

Take $x\in X$ and suppose that $x$ and $x_1$ are on the same branch on $v\in\T_f$, then there is a curve $\gamma\subset \partial v$ such that $x$ and $x_1$ are in the same component of $\S\setminus \gamma$.
By construction of $\T^X$ and according to Proposition \ref{quasidef}, the points $M_{t_n}\circ \Phi_{t_n}(x)$, $M_{t_n}\circ \Phi_{t_n}(x_1)$ and the points $M_{t_n}\circ \Phi_{t_n}(x_2)$, $M_{t_n}\circ \Phi_{t_n}(x_3)$ lies in different components of $\hat\C$ minus an annulus whose modulus converges to infinity. 
Hence 
$$M_{t_n}\circ \Phi_{t_n}(x)\to 0=M_{t_n}\circ \Phi_{t_n}(x_1),$$ 
so $x$ and $x_1$ are in the same branch on $v\in T^X$. As we said in Remark \ref{stabShishikura}, the trees $T_f$ and $T^X$ are stable with the same set of leaves, hence Lemma \ref{compartition} and Lemma \ref{corostabpart} conclude the proof.
\end{proof}

With the same reasoning we can also show the following:

\begin{lemma}\label{idty}
 After changing the labeling of the internal vertices of $T^{-1}_f$ we have $T^{-1}_f\lhd T^Y$.
\end{lemma}

We suppose for the rest of this paper that the labeling of the internal vertices of $T^X$ and $T^{-1}_f$ is such that $T_f= T^X$ and $T^{-1}_f\lhd T^Y$.


\subsubsection{Maps and local degree}

\begin{lemma} \label{lemsansnom}
Every edge of $T^{-1}_f$ considered as a subset of $T^Y$ maps onto an edge of $T_f$ considered as a subset of $T^Z$. All the elements of an edge of $T^{-1}_f$ considered as a subset of $T^Y$ have same degree.
\end{lemma}

This lemma is proven in Annexe \ref{apstbtrees}.

Let us recall the following result which is a direct consequence of the Argument Principle:

\begin{lemma}\label{prinargmult}
Let $g:\S\to\S$ be a branched cover, $X$ be a finite subset of $\S$ containing the critical values of $g$, $\Gamma$ be a multicurve on $\S\setminus X$, $\gamma^{-1}$ be a connected component of $g^{-1}(\Gamma)$ and $D$ be a connected component of $\S\setminus\gamma^{-1}$.

Then $\gamma:=g(\gamma^{-1})$ is the unique curve of $\Gamma$ such that there exist $x,x'\in X$ in the different connected components of $\S\setminus \gamma$ such that  $$\deg_{f}(\gamma^{-1})=|\card D\cap f^{-1}(x)-\card D\cap f^{-1}(x')|,$$
where the cardinals are counted with multiplicity. 
Moreover this formula works for any couple $(x,x')$ chosen in different connected components of $\S\setminus \gamma$.
\end{lemma}

In \cite{A3}(proof of Proposition 3.14), the lemma below is proven by passing to the limit the previous lemma.

\begin{lemma}\label{princarg} Let $\F:\T^Y\to\T^Z$ be a cover between trees of spheres limit of marked rational maps. Let $e$ be an edge of $ T^Y$adjacent to a vertex $v$ and $D$ denote a connected component of $T^Y\setminus \{e\}$. Then the edge $F(e)$ is the unique edge of $T^Z$ satisfying for a couple $(z,z')\in Z^2$ with $z$ and $z'$ lie in different components of $T^Z\setminus \{F(e)\}$
$$\deg_{f_v}(i_v(e))=|\card D\cap F^{-1}(z)-\card D\cap F^{-1}(z')|,$$
where $\deg_{f_v}(i_v(e))$ denotes the local degree of $f_v$ at the attaching point of $v$ and the cardinals are counting the critical points with multiplicity. Moreover this equality holds also for any such couple $(z,z')$.
\end{lemma}



\begin{proof}[Theorem \ref{translation}]
Lemma \ref{equalTXTF} proved the equality ${T_f= T^X}$.

First consider a vertex $v$ of $\T_f^{-1}$. 
Take any edge $e_{-1}$ between $v$ and $v'\in\T^{-1}_f$. As an edge of $\T^{-1}_f$, $e_{-1}$ is a closed curve in $\S$. 
Denote by $D$ a connected component of $\S\setminus e_{-1}$.
According to Lemma \ref{prinargmult}, for any $z,z'\in Z$ lying in the distinct connected components of $\S\setminus f(e_{-1})$, we have $\deg_{f}(e_{-1})=|\card D\cap f^{-1}(z)-\card D\cap f^{-1}(z')|$.

According to Lemma \ref{idty}, $e_{-1}$ can be considered as a subset of $T^Y$ and Lemma \ref{lemsansnom}
assures that $F(e_{-1})$ is the arc between $F(v)$ and $F(v')$ in $T^Z$ which corresponds to an edge $e$ of $T^X$. Comparing the formulas given in Lemma \ref{princarg} and Lemma \ref{prinargmult}, as $f$ and $F$ are equal on $Y$, we deduce that $e=f({e_{-1}})$. Hence we proved that 
$$F(e_{-1})=F([v,v'])=[F(v),F(v')]=e=f(e_{-1}).$$
And the assumptions about the local degrees follow the same way.
\end{proof}


\section{Trees and rescaling limits}\label{sectionend}

In this section we prove Theorem \ref{TOF} and Theorem \ref{TOF2}. 

Let  $f$ be an element of the sequence described in Section \ref{selfgraftingue}. According to Remark \ref{fundrem}, the map $\tau_f$ has a periodic vertex $v$ of period $k>1$,  the degree along this cycle is $4$ and a non critical branch of $\T^{-1}_f$ maps to a critical one of $\T_f$ by $\tau_f^k$ so the corresponding $f_v^{k}$ is not monomial. Hence Theorem \ref{translation} proves Theorem \ref{TOF2}.

Theorem \ref{TOF} is a translation of Theorem \ref{TOF2} via the following theorem.

 \begin{Thmbis}[\cite{A1}Theorem 2]
Let ${\bf F}$ be a portrait, let $(f_n,y_n,z_n)_n\in \Rat_{{\bf F},X}$ and let $(\F,\T^X)$ be a dynamical system of trees of spheres. Suppose that
$$\displaystyle { f}_n\overset{\lhd}{\underset{\phi_n^Y,\phi_n^Z}\longrightarrow}{ \F}.$$

If $v$ is a periodic internal vertex in a critical cycle with exact period $k$, then ${f^k_v:\S_v\to\S_v}$ is a rescaling limit corresponding to the rescaling $(\phi^Y_{n,v})_n$. 

In addition, for every $v'$ in the cycle, $(\phi^Y_{n,v'})_n$ and $(\phi^Y_{n,v})_n$ are dynamically dependent rescalings.
\end{Thmbis}

%
%
%

%

\appendix 
\makeatletter
\def\@seccntformat#1{Appendix~\csname the#1\endcsname:\quad}
\makeatother
\section{Compatibility of covers}\label{apstbtrees}

Here we prove a more general result than Lemma \ref{lemsansnom}.



\begin{definition}We say that a cover between trees of spheres ${\F':\T^{Y'}\to\T^{Z'}}$ is compatible with a cover between trees of spheres $\F:\T^Y\to\T^Z$ of same degree and write $\F'\lhd\F$ if
\begin{itemize}
\item $\T^{Y'}\lhd\T^Y$,
\item $ \T^{Z'}\lhd\T^Z$, and
\item $F'=F|_{Y}$.
\end{itemize}
\end{definition}

For two vertices $v,v'$ of a combinatorial tree $T$, we define $[v,v']$ to be the arc between $v$ and $v'$ and $]v,v'[:=[v,v']\setminus \{v,v'\}$. We define the annulus $]|v,v'|[$ to be the connected component of $T\setminus\{v,v'\}$ containing $]v,v'[$.
For a cover between trees of spheres $\F$ we denote by $\V_F$ the set of leaves which are the images of the critical leaves of $\F$.

\begin{theorem} 
Let $\F:\T^Y\to\T^Z$ be a cover between trees of spheres and $Z'$ be a set with at least three elements containing $\V_F$. There exists a unique cover between trees of spheres $\F':\T^{F^{-1}(Z')}\to\T^{Z'}$ such that $\F'\lhd\F$.
\end{theorem}

\begin{proof} We define $\T^{Z'}$ to be the tree of sphere whose set of leaves is $Z'$, whose internal vertices are the vertices of $T^Z$ separating at least three elements of $Z'$ and whose edges and their attaching points are such that $ \T^{Z'}\lhd\T^Z$. We define $Y':=F^{-1}(Z')$ and $\T^{Y'}$ whose vertices are the preimage by $F$ of the one of $T^{Z'}$. Clearly the internal vertices of $T^{Y'}$ are the vertices of $T^Y$ who separate at least three elements of $Y'$. We define again the edges and their attaching points of $\T^{Y'}$ to be such that  and $\T^{Y'}\lhd\T^Y$. We just have to prove that $F':=F|_{T^{Y}}$ is a combinatorial tree map and that the corresponding $\F':=\F|_{\T^{Y}}$ is a cover between trees of spheres.

Consider an edge of $T^{Y'}$ between two vertices $v$ and $v'$. As $\V_F\subset Z'$, the set $Y'$ contains all of the critical leaves of $F$ and the annulus $]|v,v'|[$ does not contain any critical leave. Then the result follows from Lemma \ref{lemmafunny} below.
\end{proof}

\begin{lemma}\label{lemmafunny} If Suppose that $[v,v']$ is an arc and that $]|v,v'|[$ does not contain any critical leaf, then all of the elements of $[v,v']$ have same degree and $F([v,v'])=[F(v),F(v')]$.
\end{lemma}

\begin{proof} The critical internal vertices form paths between critical leaves (see \cite{A2}). Hence, as $]|v,v'|[$ does not contain critical leaves, its only critical vertices can be on $]v,v'[$. It follows that for any $v''\in]v,v'[$, the map $f_{v''}$ has exactly two critical points, so these ones have same degree and all the $f_{v''}$ have same degree.

The image of $[v,v']$ is a connected graph that contains $F(v)$ and $F(v')$ so $[F(v),F(v')]\subset F([v,v'])$. Suppose that $F([v,v'])$ has an end $F(w)$ different than $F(v)$ and $F(v')$. Then $F(w)$  has an unique adjacent edge in $F([v,v'])$ whose attaching point is thus the image of the two only critical points of $f_{v''}$ which is no possible, so we have a contradiction.
\end{proof}

For every cover $\F$, let us denote by $\Oub_{{\bf F},{\bf F'}}(\F)$ the cover $\F'$ associated to $\F$ by the previous theorem. In fact we have proven the following result:

\begin{theorem} 
The application $\Oub_{{\bf F},{\bf F'}}:\revTC_{{\bf F}}\to\revTC_{{\bf F'}}$ is continuous.
\end{theorem}
%
%
%
%




\end{document}